\documentclass{amsart}
\usepackage{amsmath,amssymb,enumerate,amsthm}
\usepackage{comment}
\allowdisplaybreaks[2]

\makeatletter
    
    \@addtoreset{equation}{section}
  \makeatother

\newcommand{\C}{{\mathbb C}}

\newcommand{\N}{{\mathbb N}}
\newcommand{\Z}{{\mathbb Z}}
\newcommand{\R}{{\mathbb R}}

\newcommand{\mO}{{\mathcal O}}
\newcommand{\mN}{{\mathcal N}}

\newcommand{\mG}{{\mathcal G}}
\newcommand{\mU}{{\mathcal U}}

\newcommand{\fU}{{\mathfrak U}}
\newcommand{\mA}{{\mathcal A}}
\newcommand{\mB}{{\mathcal B}}
\newcommand{\acts}{\curvearrowright}
\newcommand{\ve}{\varepsilon}

\newcommand{\Ga}{\Gamma}

\newcommand{\la}{\lambda}
\newcommand{\al}{\alpha}
\newcommand{\be}{\beta}
\newcommand{\ga}{\gamma}
\newcommand{\de}{\delta}
\newcommand{\rh}{\rho}

\newcommand{\om}{\omega}
\newcommand{\vph}{\varphi}

\newcommand{\ggen}{( X^\omega )_{g\textrm{-reg}}}
\newcommand{\gfix}{( X^\omega )_g}

\newcommand{\Ggen}{( X^\omega )_{G\textrm{-reg}}}

\newcommand{\fgen}{( X^\omega )_{f\textrm{-reg}}}

\newcommand{\ffix}{( X^\omega )_f}
\newcommand{\AL}{C$^*$-algebra }
\newcommand{\ALs}{C$^*$-algebras }

\newtheorem{thm}{Theorem}[section]
\newtheorem{prop}[thm]{Proposition}
\newtheorem{cor}[thm]{Corollary}
\newtheorem{lem}[thm]{Lemma}
\newtheorem*{claim}{Claim}
\newtheorem{mainthm}{}

\theoremstyle{definition}
\newtheorem{defn}[thm]{Definition}
\newtheorem{exa}[thm]{Example}
\newtheorem{rem}[thm]{Remark}
\newtheorem{note}[thm]{Notation}
\title{On the Simplicity of C$^*$-algebras associated to multispinal groups}

\author{Keisuke Yoshida}
\email{kskyhuni@math.sci.hokudai.ac.jp}
\address{Graduate School of Science, Hokkaido University, \mbox{060-0810} Sapporo, Japan}

\date{\today}

\begin{document}

\maketitle

\begin{abstract}
We characterize the simplicity of C$^*$-algebras arising from multispinal groups.
Let $\mO_{G_{\max}}$ be the universal \AL associated to a multispinal group $G$.
We show that the invertibility of a matrix completely determines the simplicity of $\mO_{G_{\max}}$.
\end{abstract}


\section{Introduction}
Many finitely generated \ALs have been studied since the first half of 20th century when operator algebras were introduced by F. J. Murray and J. von Neumann. 
Some of them are constructed from dynamical systems.
Such C$^*$-algebras include the Cuntz algebras \cite{Cu2}.
The Cuntz algebra $\mO_n$ can be regarded as the \AL arising from the shift maps on the symbolic dynamical system over a finite set $X$ with $|X| =n$.  
Here $n \geq 2$ is a natural number.
We write $X^\om$ for the set of right infinite words over $X$.
For each $x \in X$,
the shift map $T_x \colon X^\om \to X^\om$ is given by
\[
T_x(w) := xw
\]
for $w \in X^\om$.
One often identifies each shift map with an isometry on a Hilbert space.
In this paper,
we write $S_x$ for $T_x$ when we regard $T_x$ as an operator. 
Recall that the Cuntz algebra $\mO_n$ is the universal \AL generated by $\{ S_x \}_{x \in X}$ satisfying 
\begin{equation}
\label{Cuntzintro}
S_x^*S_x=1, \quad \sum_{y \in X} S_yS_y^* =1 
\end{equation} 
for any $x \in X$.

We write $\text{Homeo}(X^\om)$ for the group of homeomorphisms on $X^\om$.
A subgroup $G \subset \text{Homeo}(X^\om)$ is said to be a \textit{self-similar group over $X$} if for any $g \in G, x \in X$ there exist $g|_{x} \in G, g(x) \in X$ with
\begin{equation}
\label{selfsimintro}
gS_x = S_{g(x)}g|_{x}.
\end{equation}
Let $\mu$ be the Bernoulli measure on $X^\om$ and let $\gfix$ be the set of fixed points of $g \in G$.

Self-similar groups play important rolls in geometric group theory.
For instance,
a self-similar group called the Grigorchuk group is known as the first example of a finitely generated group of intermediate growth \cite{Gr}.
The Grigorchuk group is a self-similar group over $\{0, 1\}$.
Iterated monodromy groups are also important examples.
They provide useful techniques to the study of some complex dynamical systems (see \cite{BN} for details).
In addition,
the following operator algebraic approach to iterated monogromy groups gives a $K$-theoretic invariant for complex dynamical systems (see \cite{Nek2}).

We study \ALs arising from self-similar groups. 
V. V. Nekrashevych has considered $*$-representations of self-similar groups.
In \cite{Nek2},
he introduced the universal \AL generated by a self-similar group $G$ and $\{S_x \}_{x \in X}$ satisfying equations (\ref{Cuntzintro}), (\ref{selfsimintro}).
In this paper,
we write $\mO_{G_{\max}}$ for the universal C$^*$-algebra.
We often assume that $G$ has a finite generation property called contracting (see Definition \ref{defcont}).
Then $\mO_{G_{\max}}$ is finitely generated \cite{Nek2}.
The \AL $\mO_{G_{\max}}$ is the main object of this paper.
Especially,
we discuss the simplicity of $\mO_{G_{\max}}$.

One can construct the groupoid $[G, X]$ of germs from self-similar group $G$ over $X$.
In \cite{Nek2},
it was shown that the universal \AL $\mO_{G_{\max}}$ is isomorphic to the full groupoid \AL $C^*([G, X])$ (see \cite{EP} for more general case). 
The groupoid $[G, X$] is minimal, effective and ample.
Thus, 
if the groupoid $[G, X]$ is Hausdorff and amenable,
then the C$^*$-algebras $C^*([G, X])$ and $\mO_{G_{\max}}$ are simple.
However,
there exists a self-similar group whose groupoid of germs is not Hausdorff. 
Such self-similar groups include the Grigorchuk group.
In \cite{CE},
L. O. Clark \textit{et al}. characterized the simplicity of \ALs of non-Hausdorff groupoids.
Using the characterization,
they show $\mO_{G_{\max}}$ is simple if $G$ is the Grigorchuk group. 
B. Steinberg and N. Szak\'{a}cs studied Steinberg algebras of non-Hausdorff groupoids arising from self-similar groups in \cite{SS}.
Combining results in \cite{CE, SS},
we obtain self-similar groups which provide non-Hausdorff amenable ample minimal effective groupoids with non-simple reduced (and full) groupoid C$^*$-algebras.
Some of such self-similar groups are multispinal groups \cite{SS}.
A multispinal group over $X$ is a self-similar group arising from two finite groups $A, B$ and a map $\Psi$ from $X$ to $\text{Aut}(A) \cup \text{Hom}(A, B)$.
Here $\text{Aut}(A)$ is the set of automorphisms on $A$ and $\text{Hom}(A, B)$ is the set of homomorphisms from $A$ to $B$.
We do not observe the definition here, see section 3 for it.
Note that every multispinal group is contracting \cite{SS}.
In addition,
the groupoid $[G, X]$ of germs is typically non-Hausdorff if $G$ is a multispinal group over $X$.
In \cite{Su},
Z. $\breve{\text{S}}$uni\'{c} constructed a family of self-similar groups which produces a generalization of the Grigorchuk group.  
Multispinal groups include this family.
For instance,
the Grigorchuk group is a multispinal group arising from two finite groups $\Z_2 \times \Z_2$ and $\Z_2$. 
Multispinal groups include other interesting self-similar groups.
For example,
Gupta-Sidki $p$-groups \cite{GS}, GGS-groups \cite{BG} 
and multi-edge spinal groups \cite{AK} are multispinal groups.
See \cite{SS} for details.

In this paper, 
we write $\mO_{G_{\min}}$ for a certain simple quotient of $\mO_{G_{\max}}$.
The quotient $\mO_{G_{\min}}$ is also introduced by V. V. Nekrashevych \cite{Nek}.
We write
\[
\langle G, X \rangle := \{S_ugS_v^* : u, v \in X^*, g \in G \}.
\]
For a representation of $\mO_{G_{\min}}$,
we consider the following state.
Define a function $\psi$ on $\langle G, X \rangle$ by 
\[
\psi(S_ugS_v^*) = \de_{u, v}\mu(\gfix)
\]
Here $\de$ is the Kronecker delta.
If $G$ is contracting,
then the function $\psi$ extends to KMS states $\psi_{\min}$ and $\psi_{\max}$ on $\mO_{G_{\min}}$ and $\mO_{G_{\max}}$,
respectively \cite{BL, LR, Y}. 
Here note that we have canonical embeddings of $G$ into $\mO_{G_{\min}}$ and $\mO_{G_{\max}}$.
The uniqueness of KMS states is studied in \cite{BL, LR, Y} but we do not need the uniqueness in this paper. 

We discuss when the universal \AL $\mO_{G_{\max}}$ is simple under the assumption that $G$ is a multispinal group and $[G, X]$ is amenable. 
Note that we do NOT assume $[G, X]$ is Hausdorff.
If $[G, X]$ is amenable and $C^*_r([G, X])$ is isomorphic to $\mO_{G_{\min}}$,
then the universal \AL $\mO_{G_{\max}}$ is simple.  
To get an isomorphism between $\mO_{G_{\min}}$ and $C^*_r([G, X])$,
we compare two Hilbert spaces.
One of the Hilbert spaces is the GNS space $H_{\psi_{\min}}$ of $\psi_{\min}$.
The reduced \AL $C_r^*([G, X])$ acts on the other one.
For the comparison between two Hilbert spaces,
we consider the full group C$^*$-algebra $C^*(A)$ of the finite subgroup $A$.
We identify $C^*(A)$ with finite dimensional C$^*$-subalgebra of $\mO_{G_{\max}}$ generated by $A$.
As a consequence,
we obtain the characterization of the simplicity as follows (see Theorem \ref{themsec5}).

\begin{mainthm}
Let $G = G(A, B, \Psi)$ be a multispinal group over $X$.
Assume that $[G, X]$ is amenable.
Then the following conditions are equivalent.
\begin{enumerate}
\renewcommand{\labelenumi}{\textup{(\roman{enumi})}}
\item
The restriction of $\psi_{\max}$ to $C^*(A)$ is faithful.
\item
The $|A| \times |A|$ matrix $[\psi(a_1^{-1} a_2)]_{a_1, a_2 \in A}$ is invertible.
\item
The universal \AL $\mO_{G_{\max}}$ is simple.
\end{enumerate}
\end{mainthm}

Under the assumption of the Main Theorem,
if one of the equivalent conditions holds,
then $\mO_{G_{\max}}$ is a Kirchberg algebra.
In addition,
one can compute $\psi(a) = \mu((X^\om)_a)$ for each $a \in A$ by solving a linear equation. 
It is not hard to check that the condition (ii) holds for the Grigorchuk group (see Example \ref{simplegri}).
Thus,
as a corollary of the Main Theorem,
we obtain a different proof of \cite[Theorem 5.22]{CE} which claims that $\mO_{G_{\max}}$ is a Kirchberg algebra for the Grigorchuk group $G$. 

\section{Preliminaries}
In this section,
we recall the definition of self-similar groups and related topics.

\subsection{Self-similar groups}
We begin with notations used in this paper. 
\begin{note}
Let $X$ be a finite set $X$ with $|X| \geq 2$. 
In this paper,
$X^*$ denotes the set of finite words over the alphabet $X$.
In other words,
$X^* = \bigsqcup_{n \in \N} X^n $
(we define $X^0 := \{ \emptyset \}$).
Write $X^\om$ for the set of unilateral infinite words over $X$.
For $w \in X^\om$ and $n \in \N$,
$w^{(n)} \in X^n $ denotes the first $n$ letters of $w$.

For a finite word $u \in X^n$,
we write $|u|$ for the length of the word $u$,
namely $|u| =n$.
Take a subset $P \subset X^\om$.
We define 
\[
uP := \{ uw : w \in P \} \subset X^\om .
\]

We identify $X^\N$ with $X^\om$.
Thus,
it is equipped with the product topology of the discrete sets $X$.
Note that $X^\om$ is homeomorphic to the Cantor space.
Write $\text{Homeo}(X^\om)$ for the group of homeomorphisms on $X^\om$.
\end{note}

\begin{defn}(\cite[Definition 2.1]{Nek})
A subgroup $G$ of $\text{Homeo}(X^\om)$ is said to be a \textit{self-similar group over $X$}
if
for every $g \in G$ and $x \in X$ there exist $h \in G $ and $v \in X $
with
\begin{equation}
g(xw) = vh(w)
\label{selfsim}
\end{equation}
for any $w \in X^\om $.
\end{defn}

In this paper, 
$G$ always denotes a self-similar group over $X$ where $X$ is a finite set with at least 2 elements.  

\begin{rem}
Using the equation (\ref{selfsim}) several times,
we see that for every $n \in \N$,
$g \in G$ and $u \in X^n$ there exist $h \in G $ and $v \in X^n $ with 
\[
g(uw) = vh(w)
\]
for any $w \in X^\om $.
A direct calculation shows that $h$ and $v$ are uniquely determined by $g$ and $u$. 
Write $h=g|_u$ and $v=g(u)$.
\end{rem}

The following finitely generated group was first introduced in \cite{Gr}.

\begin{exa}(\cite[Example 2.5]{Nek})\label{Gri}
Let $X=\{ 0, 1\}$.
Consider homeomorphisms $a, b, c, d$ on $X^\om$ given by
\[
a(0w) = 1w , \ a(1w) = 0w, 
\]
\[
 b(0w) = 0a(w), \ b(1w) = 1c(w),
\]
\[
c(0w) = 0a(w) , \ c(1w) = 1d(w), 
\]
\[
 d(0w) = 0w , \ d(1w) = 1b(w)
\]
for $w \in X^\om$.
Let $G$ be the subgroup of $\text{Homeo}(X^\om)$ generated by $a, b, c, d$.
The above relations imply that $G$ is a self-similar group.
The group $G$ is called the \textit{Grigorchuk group}. 
\end{exa}
For more details and examples,
see \cite{LR, Nek2, SS}.
We often keep the following regularity in our mind in the study of self-similar groups and their operator algebras.

\begin{defn}(\cite[Definition 4.1]{Nek})\label{defreg}
Let $G$ be a self-similar group over $X$ and fix $g \in G$.
An element $w \in X^\om$ is said to be \textit{$g$-regular} if either $g(w) \neq w$ or there exists an open neighborhood of $w$ consisting of fixed points of $g$.
Let $\ggen$ be the set of all $g$-regular points.
Write $\Ggen := \bigcap_{g \in G} \ggen$.
An element $w \in \Ggen$ is said to be \textit{$G$-regular}.
\end{defn} 

Recall notations.
Take any $u \in X^*$ and let $T_u$ be the shift map on $X^\om$ given by $w \mapsto uw$.
Let $T_u^*$ denote the local inverse map of $T_u$ defined on the range of $T_u$.
Write 
\[
\langle G, X \rangle := \{ T_{u}gT_{v}^* : u, v \in X^*, g \in G\}.
\]

\begin{defn}(\cite[Definition 9.1]{Nek})
Let $G$ be a self-similar group over $X$ and fix $f \in \langle G, X \rangle$.
An element $w \in X^\om$ is said to be \textit{$f$-regular} if either $f$ is not defined on $w$ or does not fix $w$ or there exists an open neighborhood of $w$ consisting of fixed points of $f$.
We write $(X^\om)_{f\text{-reg}}$ for the set of $f$-regular points.
Also write $\Ggen ^\text{S} := \bigcap_{f \in \langle G, X \rangle} (X^\om)_{f\text{-reg}}$.
An element in $\Ggen ^\text{S}$ is said to be \textit{strictly $G$-regular}.
\end{defn}

From now on,
we always assume that $G$ is a countable self-similar group.
The following remark tells us why we need the countability.

\begin{rem}
It is not hard to show that $\fgen$ is an open dense subset of $X^\om$ for any $f \in \langle G, X \rangle$. 
Thus, 
the Baire category theorem implies that $\Ggen ^\text{S}$ is dense in $X^\om$ by the countability of $G$.
Note that $\Ggen ^\text{S}$ is a $\langle G, X \rangle$-invariant set.
\end{rem}

For $f \in \langle G, X \rangle$, 
let $\ffix$ be the set of fixed points of $f$.
By definition,
$X^\om \backslash \fgen \subset \ffix$.

Let $\mu$ be the product measure of the uniform probability measures on $X$'s (we identify $X^\om$ with $X^{\N}$).
We often refer to the measure $\mu$ as the Bernoulli measure.
For any $u \in X^*$,
we have $\mu (uX^\om) = | X |^{-|u|}$. 
We recall the following facts in \cite{Y}.

\begin{lem}\label{gen}\textup{(\cite[Lemma 2.10]{Y})}
Let $G$ be a countable self-similar group over $X$.
Then the following conditions are equivalent.
\begin{enumerate}
\renewcommand{\labelenumi}{\textup{(\roman{enumi})}}
\item
$\mu((X^\om)_{G\textup{-reg}}) =1$.
\item
$\mu((X^\om)_{g\textup{-reg}}) = 1$ for any $g \in G$.
\item
$\mu((X^\om)_{G\textup{-reg}} ^\textup{S})=1$.
\item
$\mu((X^\om)_{f\textup{-reg}}) =1$ for any $f \in \langle G, X \rangle$.
\end{enumerate} 
\end{lem}

\begin{thm}\textup{(\cite[Theorem 2.12]{Y})}
For any countable self-similar group $G$ over $X$,
one has either $\mu((X^\om)_{G\textup{-reg}}) =1$ or $\mu((X^\om)_{G\textup{-reg}}) =0$.
\end{thm}

For an example of a self-similar group $G$ over $X$ with $\Ggen =1$,  
we recall the following property.  

\begin{defn}\label{defcont}(\cite[Definition 2.2]{Nek2})
A self-similar group $G$ over $X$ is said to be \textit{contracting} if there exists a finite set $\mN \subset G$ with the following condition:

For any $g \in G$ there exists $n \in \N$ satisfying $g|_v \in \mN$ for any $v \in X^*$ with $|v| >n$.  

If a self-similar group $G$ is contracting,
the smallest finite set of $G$ satisfying the above condition is said to be the \textit{nucleus of $G$}.
\end{defn} 

The proposition below is included in the proof of \cite[Theorem 7.3 (3)]{LR}

\begin{prop}\label{cont}
Let $G$ be a contracting countable self-similar group over $X$.
Then we have
$\mu((X^\om)_{G\textup{-reg}})=1$.
\end{prop}  

It is not hard to show that the Grigorchuk group is contracting.
We observe more examples of contracting self-similar groups in the next section.

\subsection{\ALs arising from self-similar groups} 
We review the \ALs introduced by Nekrashevych constructed from self-similar groups in \cite{Nek, Nek2}. 

\begin{defn}(\cite[Definition 3.1]{Nek})
Define $\mO_{G_{\max}}$ to be the universal \AL generated by $G$ (we assume that every relation in $G$ is preserved) and $\{ S_x : x \in X \}$ satisfying the following relations for any $x \in X$ and $g \in G$:
\begin{equation}
g^*g = gg^* = 1,
\end{equation}

\begin{equation}
S_x^*S_x = 1,
\label{isometry}
\end{equation}

\begin{equation}
\sum_{y \in X}S_yS_y^* = 1,
\label{Cuntz}
\end{equation}

\begin{equation}
\label{similar}
gS_x = S_{g(x)}g|_x.
\end{equation}
\end{defn}

The universal \AL $\mO_{G_{\max}}$ admits a nonzero representation.
We recall the following representations.

\begin{rem}\label{representation}
For each $w \in X^\om$,
$\langle G, X \rangle(w)$ denotes the $\langle G, X \rangle$-orbit of $w$,
i.e. 
\[
\langle G, X \rangle(w) = \{ f(w) : f \in \langle G, X \rangle, w \in \text{Dom}f \}.
\]
Let $l^2(\langle G, X \rangle(w))$ be the set of $l^2$-functions on $\langle G, X \rangle(w)$.
Write $\de_{w_1}$ for the characteristic function of $w_1 \in \langle G, X \rangle(w)$.
We naturally identify each shift map $T_x$ with an isometry on $l^2(\langle G, X \rangle(w))$ given by $T_x(\de_{w_1}) :=\de_{xw_1}$. 
Similarly,
we identify each $g \in G$ with a unitary on $l^2(\langle G, X \rangle(w))$ given by $g(\de_{w_1}) = \de_{g(w_1)}$.
Thanks to the universality of $\mO_{G_{\max}}$,
we get a representation $\pi_w$ of $\mO_{G_{\max}}$ on $l^2(\langle G, X \rangle(w))$ given by 
\[
\pi_w(S_x) =T_x, \ \pi_w(g) =g
\]
for any $x \in X, g \in G$. 
In particular,
$\mO_{G_{\max}}$ is nonzero.
\end{rem}

For $u=x_1x_2 \cdots x_n \in X^n$,
we write $S_u := S_{x_1}\cdots S_{x_n}$.
Equations (\ref{isometry}), (\ref{Cuntz}) and Remark \ref{representation} imply that $\mO_{G_{\max}}$ contains the Cuntz algebra $\mO_{|X|}$.

We recall the following result due to V. V. Nekrashevych.

\begin{thm}\textup{(\cite[Theorem 3.3]{Nek2})}\label{uniqueNek}
Let $\rh$ be a unital representation of $\mO_{G_{\max}}$ on a nonzero Hilbert space.
Then for any $w_0 \in \Ggen ^\textup{S}$ and $a \in \mO_{G_{\max}}$,
we have 
\[
\| \pi_{w_0}(a) \| \leq \| \rh(a) \|.
\]
\end{thm}

The above theorem implies that the \AL $\mO_{G_\text{min}} :=  \pi_{w_0}(\mO_{G_{\max}})$
does not depend on the choice of $w_0 \in \Ggen ^\text{S}$ up to canonical isomorphisms.

We identify $\langle G, X \rangle$ with 
\[
\{ S_{u}gS_{v}^* : u, v \in X^*, g \in G\} \subset \mO_{G_{\max}}.
\]
For  each $w \in X^\om$,
the restriction of $\pi_{w}$ to $\langle G, X \rangle$ is injective.
Hence,
we also write $f$ for $\pi_{w_0}(f) \in \mO_{G_{\min}}$ where $w_0 \in \Ggen ^\text{S}$ and $f \in \langle G, X \rangle$. 
Let $\C\langle G, X \rangle$ be the $*$-subalgebra of $\mO_{G_{\max}}$ generated by $G$ and $\{S_x \}_{x \in X}$.
The self-similarity implies
\[
\C \langle G, X \rangle = \text{span} \{ S_{u}gS_{v}^* : u, v \in X^*, g \in G\}.  
\]

To understand $\mO_{G_{\min}}$, 
we observe a Hilbert bimodule.
Let $A_{G_{\min}}$ be the C$^*$-subalgebra of $\mO_{G_{\min}}$ generated by $G$.
Define a subspace $\Phi$ of $\mO_{G_{\min}}$ by 
\[
\Phi :=\left\{ \sum_{ x \in X } S_xa_x \ \middle| \ a_x \in A_{G_\text{min}} \right\}.
\]
We regard $\Phi$ as a right $A_{G_\text{min}}$-module with the basis $\{S_x\}_{x \in X}$.
Define an $A_{G_\text{min}}$-valued inner product on $\Phi$ by
\[
\langle \sum_{x \in X}S_x a_x , 
\sum_{x \in X }S_xb_x \rangle
:= \sum_{x \in X }a_x^* b_x
\]
where $a_x ,
b_x \in A_{G_\text{min}}$ for any $x \in X$. 
Consider the left action of $A_{G_{\min}}$ on $\mO_{G_{\min}}$ arising from the multiplication. 
The left action provides a left $A_{G_{\min}}$-module structure on $\Phi$.
Thus,
we obtain a Hilbert $A_{G_\text{min}}$-bimodule. 
We also write $\Phi$ for this Hilbert $A_{G_\text{min}}$-bimodule. 

The following definition is almost the same as \cite[Definition 6.1]{Nek}.

\begin{defn}(\cite[Definition 6.1]{Nek})
Define $\mO_{G_{\min}}'$ to be the universal \AL generated by $A_{G_{\min}}$ and $\{ S_x : x \in X \}$ satisfying equations (\ref{isometry}), (\ref{Cuntz}) and the following relation for any $a \in A_{G_{\min}}$ and $x \in X$:
\begin{equation}\label{inner}
aS_x = \sum_{y \in X}S_y\langle S_y, aS_x \rangle.
\end{equation}
Here the inner product is the $A_{G_{\min}}$-valued one defined as above.
\end{defn} 

The same argument as \cite[Theorem 8.3]{Nek} implies the simplicity of $\mO_{G_{\min}}'$.  
\begin{thm}\textup{(\cite[Theorem 8.3]{Nek})}\label{simple}
The universal \AL $\mO_{G_{\min}}'$ is unital, purely infinite and simple.
\end{thm}

Let $\mO(\Phi)$ be the Cuntz--Pimsner algebra arising from the $A_{G_{\min}}$-bimodule $\Phi$.
In fact, 
$\mO_{G_{\min}}'$ is isomorphic to $\mO(\Phi)$ and also isomorphic to $\mO_{G_\text{min}}$.
Indeed,
using the simplicity from Theorem \ref{simple},
one can construct injective universal surjections from $\mO_{G_{\min}}'$ onto both $\mO_{G_{\min}}$ and $\mO(\Phi)$. 
Hence $\mO_{G_\text{min}}'$ is isomorphic to $\mO_{G_\text{min}}$ and $\mO(\Phi)$.

We fix a definition and some notations related to the canonical gauge actions and KMS states.

\begin{defn}\label{defKMS}
Let $D$ be a C$^*$-algebra.
Fix a group action $\al \colon \R \acts D$.
Assume that the map $t \mapsto \al_t(c)$ defined on $\R$ is norm-continuous for any $c \in D$.
An element $d \in D$ is said to be \textit{$\al$-analytic} if the map $t \mapsto \al_t(d)$ extends to an analytic map $z \mapsto \al_z(d)$ on $\C$.

In addition,
fix a nonzero real number $\be$.
A state $\vph$ on $D$ is said to be a \textit{$(\beta, \al)$-KMS state} if 
\[
\vph(d c) = \vph(c \al_{i \beta}(d) ) 
\]
for any $c \in C$ and any $\al$-analytic element $d \in C$.
\end{defn}

Under the assumption in Definition \ref{defKMS},
we write
\[
D^{\al} := \{ d \in D : \al_t(d) = d \ \text{for any} \ t \in \R \}.
\]

Let $\Ga$ be the canonical gauge action of $\R$ on $\mO_{G_{\max}}$ given by 
\[
\Ga_t(g):=g  \ \textrm{and}  \ \Ga_t(S_x):=\exp(it)S_x
\]
for $g \in G$,
$t \in \R$ and $x \in X$.
We also define the canonical gauge action $\Ga$ on $\mO_{G_{\min}}$ (we use the same symbol as there is no confusion) by
\[
\Ga_t(a):=a  \ \textrm{and}  \ \Ga_t(S_x):=\exp(it)S_x
\]
for $a \in A_{G_{\min}}$,
$t \in \R$ and $x \in X$.
Using the universalities,
one can check the above gauge actions are well-defined. 
We write
\[
\langle G, X \rangle ^\Ga : = \{ S_u g S_v^* \in \langle G, X \rangle : |u|= |v| \} \subset \langle G, X \rangle.
\]
Note that 
\[
\mO_{G_{\max}} ^\Ga= \overline{\text{span}}\langle G, X \rangle ^\Ga.
\]
Define a linear functional $\psi$ on $\C \langle G, X \rangle$ by 
\[
\psi(S_{u}gS_{v}^*) := \de_{u, v} |X|^{-|u|}\mu(\gfix)
\]
where $\de$ is the Kronecker's delta.
We have $\psi(f) = \mu(\ffix)$ for any $f \in \langle G, X \rangle$ by \cite[Remark 2.9]{Y}. 

\begin{lem}
Let $G$ be a self-similar group over $X$ with $\mu(\Ggen)=1$. 
Fix $w_0 \in \Ggen ^\text{S}$.
Then the map
\[
\psi '(\pi_{w_0}(f)) :=\psi(f) \ ; \quad f \in \langle G, X \rangle
\]
extends to a linear functional $\psi '$ on $\pi_{w_0}(\C \langle G, X \rangle)$.
\end{lem}

\begin{proof}
Take any $\eta \in \C \langle G, X \rangle$ with $\pi_{w_0}(\eta)=0$.
We prove $\psi(\eta)=0$.
Write 
\[
\eta = \sum_{f \in \langle G, X \rangle} \ga_f f 
\]
where $\ga_f =0$ for all but finitely many $f \in \langle G, X \rangle$.
One has 
\[
\psi(f) = \mu(\ffix)=  \int_{X^\om}\langle \de_w, f(\de_w) \rangle d\mu(w) .
\]
for each $f$.
Each inner product above is defined on $l^2(\langle G, X \rangle (w))$.
The assumption and Lemma \ref{gen} imply $\mu(\Ggen ^\text{S}) = 1$.
Then a calculation shows
\begin{equation*}
\begin{split}
\psi(\eta) =\sum_{f \in \langle G, X \rangle} \ga_f \int_{X^\om} \langle \de_w, f(\de_w) \rangle d\mu(w)
&=\sum_{f \in \langle G, X \rangle} \ga_f \int_{\Ggen ^\text{S}} \langle \de_w, f(\de_w) \rangle d\mu(w) \\
&= \int_{\Ggen ^\text{S}} \langle \de_w, (\pi_w (\eta))(\de_w) \rangle d\mu(w) =0
\end{split}
\end{equation*}
since $\pi_w(\eta) =0$ for any $w \in \Ggen ^\text{S}$.
This proves the claim.

\end{proof}

In the case $\mu(\Ggen)=1$,
the linear functionals $\psi$ and $\psi '$ extend to unique $(\log|X|, \Ga)$-KMS states on $\mO_{G_{\max}}$ and $\mO_{G_{\min}}$,
respectively. 
See sections 6, 8 of \cite{BL} for the uniqueness of KMS states on Toeplitz type \AL associated with self-similar groups.
  
\begin{thm}\label{main2}\textup{(\cite{BL}, \cite[Theorems 3.11, 3.13]{Y})}
Assume that $G$ is a countable self-similar group over $X$ with $\mu((X^\om)_{G\textup{-reg}})=1$.
Then the following statements hold:
\begin{enumerate}
\renewcommand{\labelenumi}{\textup{(\roman{enumi})}}
\item
There exists a unique $(\log|X|, \Ga)$-KMS state $\psi_{\max}$ on $\mO_{G_{\max}}$.
\item
There exists a unique $(\log|X|, \Ga)$-KMS state $\psi_{\min}$ on $\mO_{G_{\min}}$.
\item
For any $u, v \in X^*, g \in G$,
we have
\[
\psi_{\max}(S_u g S_v^*) = \psi_{\min}(S_u g S_v^*) =\de_{u, v} \mu(\gfix).
\]
\end{enumerate}
\end{thm}

The uniqueness in Theorem \ref{main2} is not used in the next section.
  
\section{Groupoid approaches and simplicity of $\mO_{G_{\max }}$} 
In this section, 
we observe the groupoids arising from self-simlar groups and their groupoid C$^*$-algebras. 
For some contracting self-similar groups,
we discuss the simplicity of $\mO_{G_{\max }}$ through the groupoid C$^*$-algebras.

\subsection{Groupoids arising from self-similar groups}  
 
As the first part of this section, we recall the definitions and fix notations.
A small category $\mG$ whose any morphism has the inverse is said to be a \textit{groupoid}.
We identify a groupoid $\mG$ with the set of morphisms and identify the objects with the identity morphisms on them. 
We write the set of objects $\mG^{(0)} \subset \mG$ and refer to it as the \textit{unit space of $\mG$}.
Define two maps $s, r$ from $\mG$ onto $\mG^{(0)}$ by
\[
s(\ga) = \ga^{-1} \ga \ \text{and} \ 
r(\ga) = \ga \ga^{-1}
\]
for $\ga \in \mG$.
The map $s$ is called the \textit{source map} and $r$ is called the \textit{range map}.

A \textit{topological groupoid} is a groupoid with a topology such that the multiplication operator and the inverse operator are continuous.
The source and range maps are continuous on any topological groupoid.
A topological groupoid is said to be \textit{\'{e}tale} (or $r$-discrete) if the unit space is locally compact with respect to the relative topology of $\mG$ and the source map is a local homeomorphism.
 
A groupoid of germs of local homeomorphisms on a topological space is a good example of a groupoid.
We observe the groupoids of germs arising from self-similar groups. 
Such groupoids have been introduced in \cite{Nek}.

Let $(\langle G, X \rangle \times X^{\om})'$ be the subset of $\langle G, X \rangle \times X^{\om}$ given by 
\[
(\langle G, X \rangle \times X^{\om})' := \{ (f, w) \in \langle G, X \rangle \times X^{\om} : w \in \text{Dom} f \}
\] 
We define an equivalence relation on $(\langle G, X \rangle \times X^{\om})'$ as follows.
Pairs $(f_1, w_1), (f_2, w_2)$ in $(\langle G, X \rangle \times X^{\om})'$ are equivalent if $w_1 = w_2$ and $f_1 = f_2$ on a neighborhood of $w_1 = w_2$.
The quotient set is denoted by $[G, X]$. 
We write $[f, w]$ for the equivalence class represented by $(f, w)$.

The set $[G, X]$ forms a groupoid under the following operations.
The multiplication is given by
\[
[f_1, w_1] \cdot [f_2, w_2] := [f_1f_2, w_2]
\]
when $w_1 = f_2(w_2)$.
The inverse is given by
\[
[f, w]^{-1} := [f^{-1}, f(w) ].
\] 
We also equip $[G, X]$ with the topology generated by sets of the form:
\[
\mU_{U, f} = \{ [f, w] : w \in U \}
\]
where $f \in \langle G, X \rangle$ and $U$ is an open subset of $\text{Dom}f$.
It is not hard to check that $[G, X]$ is \'{e}tale.

The topological groupoid $[G, X]$ might not be Hausdorff.
For example,
let $X = \{0, 1\}, G = \langle a, b, c, d \rangle$ be the Grigorchuk group (see Example \ref{Gri} for definition).  
Define 
\[
1^{\infty} :=11111 \cdots \in X^\om .
\] 
We write $e$ for the unit of $G$.
Then $[e, 1^\infty] \neq [d, 1^\infty]$ by definition but no open sets separate them.
To check this,
note that $d = e$ on $1^{3m}0X^\om$ for any nonnegative integer $m$.
For any open neighborhood $U, V$ of $1^\infty$,
there exists $n \in \N$ with
$1^{3n} X^\om \subset U$ and $1^{3n} X^\om \subset V$.
We have
\[
\mU_{1^{3n} 0X^\om, e} \subset \mU_{1^{3n} X^\om, e} \subset \mU_{U, e}
\]
and
\[
\mU_{1^{3n} 0X^\om, d} \subset \mU_{1^{3n} X^\om, d} \subset \mU_{V, d}.
\]
In addition,
we also have $\mU_{1^{3n} 0X^\om, e} = \mU_{1^{3n} 0X^\om, d} \neq \emptyset$.
Therefore, 
$\mU_{U, e} \cap \mU_{V, d} \neq \emptyset$.
This proves the claim.

Next we review the groupoid C$^*$-algebra.
In this context,
one often assumes that groupoids are Hausdorff.
In this paper,
however, 
we do NOT assume the Hausdorffness to treat examples in subsection 3.2.
The definitions of groupoid \ALs without Hausdorffness are introduced by A. Connes in \cite{C1}.

We again consider a (not necessarily Hausdorff) \'{e}tale groupoid $\mG$.
Let $\fU \subset \mG$ be an open Hausdorff subset.
The set of continuous functions on $\fU$ with compact support is denoted by $C_c(\fU)$.
Let $\text{Funct}(\mG)$ be the set of all functions on $\mG$.
For $\eta \in C_c(\fU), \ga \notin \fU$,
we define $\eta(\ga)=0$.
This gives an embedding of $C_c(\fU)$ into $\text{Funct}(\mG)$.
We then define  
\[
C(\mG) := \text{span} \bigcup_{\fU} C_c(\fU) \subset \text{Funct}(\mG) 
\]
where the union is taken over all Hausdorff open subsets $\fU$ of $\mG$. 
Note that a function in $C(\mG)$ might not be continuous on $\mG$.
   
We define the multiplication and involution operators on $C(\mG)$ by
\[
\eta^{*}(\ga) := \overline{\eta(\ga^{-1})}, \quad (\eta_1 * \eta_2)(\ga) := \sum_{\ga_1\ga_2 = \ga} \eta_1(\ga_1)\eta_2(\ga_2)
\]
for any $\eta, \eta_1, \eta_2 \in C(\mG)$.
To introduce a C$^*$-norm on $C(\mG)$,
we fix the following notation.
Let $\al \in \mG^{(0)}$ and define
\[
\mG_\al := \{ \ga \in \mG : s(\ga) = \al \}.
\]
Consider the $*$-representation $\la_\al$ of $C(\mG)$ on $l^2(\mG_\al)$ given by
\[
(\la_\al(\eta)\zeta )(\ga) = \sum_{\ga_1\ga_2 = \ga} \eta(\ga_1)\zeta(\ga_2)
\]
for $\ga \in \mG_\al$.
It is not hard to check that 
\[
\| \eta\|_{\text{red}} = \sup_{\al \in \mG^{(0)}} \| \la_{\al}(\eta) \|
\]
is a C$^*$-norm on $C(\mG)$. 
The completion of $C(\mG)$ with respect to $\| \cdot \|_\text{red}$ is denoted by $C^*_r(\mG)$. 
The \AL $C^*_r(\mG)$ is said to be the \textit{reduced groupoid C$^*$-algebra of $\mG$}.

For $\eta \in C(\mG)$,
we define the universal norm by 
\[
\| \eta \|_u := \sup \{ \| \rho(\eta) \| : \rho \ \text{is a $*$-representation of} \ C(\mG) \ \text{on a Hilbert space} \}.
\]
This indeed defines a C$^*$-norm on $C(\mG)$ (see \cite[Lemma 1.2.3]{Ko}).  
The completion of $C(\mG)$ with respect to $\| \cdot \|_u$ is denoted by $C^*(\mG)$. 
The \AL $C^*(\mG)$ is said to be the \textit{full groupoid $C^*$-algebra of $\mG$}.  

We again consider the groupoid of germs arising from a self-similar group $G$ over a finite set $X$.
The universality of $\mO_{G_{\max }}$ provides a surjection from $\mO_{G_{\max}}$ onto $C^*([G, X])$ given by
\[
\langle G, X \rangle \ni f \mapsto 1_{\mU_{\text{Dom}f, f}} \in C([G, X]) \subset C^*([G, X])
\]
where $1_{\mU_{\text{Dom}f, f}}$ is the characteristic function on $\mU_{\text{Dom}f, f}$.   
It was shown by V. V. Nekrashevych that this surjection is in fact an isomorphism.

\begin{thm}\textup{(\cite[Theorem 5.1]{Nek}, \cite[Corollary 6.4]{EP})}\label{isomNek1}
The full groupoid \AL $C^*([G, X])$ is isomorphic to $\mO_{G_{\max }}$.
\end{thm}

Thanks to the theorem,
one can identify $f$ with $1_{\mU_{\text{Dom}f, f}}$.
Moreover,
by identifying $C^*([G, X])$ with $\mO_{G_{\max}}$, 
we also consider the gauge action $\Ga$ defined in section 2 on $C^*([G, X])$ and its quotient $C^*_r([G, X])$.
Let 
\[
[G, X]^\Ga := \{ [S_u g S_v^*, w] \in [G, X] : |u| = |v| \}
\]
then the restriction of the isomorphism in Theorem \ref{isomNek1} gives the following identification.

\begin{thm}\textup{(\cite[Theorem 5.3]{Nek})}\label{isomNek2}
$C^*([G, X]^\Ga) \simeq C^*([G, X])^\Ga \simeq \mO_{G_{\max }}^\Ga$.
\end{thm}

At the rest of this section,
we discuss the simplicity of $C^*([G, X])$ and $C_r^*([G, X])$ rather than $\mO_{G_{\max }}$.
If $C^*([G, X])$ and $C^*_r([G, X])$ are not canonically isomorphic, 
then $C^*([G, X])$ is not simple.
Thus, 
we only consider the case they are isomorphic.

It is a well known fact that the full and reduced \AL of an amenable groupoid are canonically isomorphic. 
One can find this fact for Hausdorff groupoids in \cite{AR} or \cite{R}.
The same argument provides the isomorphism in the non-Hausdorff case. 
We recall sufficient conditions of the amenability of $[G, X]$ in \cite{EP} and \cite{Nek}.

\begin{thm}\label{ameimpame}\textup{(\cite[Corollary 10.18]{EP})}
If a self-similar group $G$ over $X$ is amenable as a discrete group, 
then $[G, X]$ is amenable. 
\end{thm} 

For the other sufficient condition,
we recall the following definition.

\begin{defn}\textup{(\cite[Definition 2.3]{Nek})}\label{contreplame}
A self-similar group $G$ over $X$ is said to be \textit{self-replicating} if for any $x, y \in X, h \in G$ there exists $g \in G$ with $g(x) =y$ and $g|_x = h$.  
\end{defn}
 
\begin{thm}\textup{(\cite[Theorem 5.6]{Nek})}\label{ame2}
If a self-similar group $G$ over $X$ is contracting and self-replicating, 
then $[G, X]$ is amenable. 
\end{thm} 

If $[G, X]$ is amenable and Hausdorff, 
then $C^*([G, X])$ is simple by standard arguments (see \cite{Nek}).
However,
there exists a self-similar group $G$ over $X$ such that $[G, X]$ is non-Hausdorff amenable groupoid but the $C^*([G, X])$ is not simple.
We observe this later.  
In the next subsection,
we consider a class of self-similar groups whose groupoids are not Hausdorff.
We discuss the simplicity of their C$^*$-algebras. 

\subsection{Multispinal groups}
From now on we consider a class of self-similar groups called multispinal groups introduced in \cite{SS}.
We first recall the construction.

Let $A, B$ be finite groups and let $B$ act freely on a finite set $X$.
Write $\text{Aut}(A)$ and $\text{Hom}(A, B)$ for the set of automorphisms of $A$ and the set of homomorphisms from $A$ to $B$, respectively. 
Consider a map $\Psi$ from $X$ to $\text{Aut}(A) \cup \text{Hom}(A, B)$. 
Define 
\[
\mA := \Psi(X) \cap \text{Aut}(A), \ \mB:=\Psi(X) \cap \text{Hom}(A, B).
\]
Set 
\[
\mB \cdot \mA := \mB \cup \bigcup_{n \geq 2} \{ \la_1\circ \la_2 \circ \cdots \circ \la_n : \la_1 \in \mB, \la_i \in \mA \ \text{for} \ 2\leq i \leq n \} \subset \text{Hom}(A, B).  
\]  
We assume $\mA \neq \emptyset, \mB \neq \emptyset$ and
\[
\bigcap_{\la \in \mB \cdot \mA} \ker \la = \{1_{A}\}.
\]  
Here $1_A$ is the unit of $A$.
We define actions of $A$ and $B$ on $X^{\om}$ by
\[
a(xw) := x(\Psi(x)(a))(w), \ b(xw) := b(x)w
\]
where $a \in A, b \in B, x \in X, w \in X^{\om}$.
The assumption implies that these two actions are faithful.
Hence,
we identify two finite groups $A, B$ with their images in $\text{Homeo}(X^\om)$, 
respectively. 
Let $G$ be the subgroup of $\text{Homeo}(X^\om)$ generated by $A$ and $B$.
The group $G$ is said to be a \textit{multispinal group} over $X$.
Let $G=G(A, B, \Psi)$ denote the multispinal group arising from two finite groups $A$, $B$ and a map $\Psi$.
Put
\[
Y:= \Psi^{-1}(\textup{Hom}(A, B)) \subset X.
\]

\begin{rem}
\label{multiassume}
Let $G=G(A, B, \Psi)$ be a multispinal group over $X$.
By definition,
the multispinal group $G$ is a self-similar group.
In fact,
the multispinal group $G$ is always contracting and the nucleus coincides with 
\[
A \cup \bigcup_{y \in Y} \Psi(y)(A) \subset A \cup B. 
\]
See \cite[Proposition 7.1]{SS} for this fact.
Assume that the action of $B$ on $X$ is transitive and 
\[
\bigcup_{y \in Y} \Psi(y)(A) =B.
\]
Then $G$ is self-replicating (thus $G$ is infinite) and the nucleus of $G$ coincides with $A \cup B$ (See section 7 of \cite{SS}).
Therefore,
$[G, X]$ is amenable by Theorem \ref{ame2}.
\end{rem}

The Grigorchuk group is an example of a multispinal group. 
\begin{exa}\label{ortgri}
Let $A = \Z_2 \times \Z_2, B = X = \Z_2$.
Consider the left translation action $B \acts X$.
We define $\Psi(0) \in \text{Hom}(\Z_2 \times \Z_2, \Z_2)$ and $\Psi(1) \in \text{Aut}(\Z_2 \times \Z_2)$ to be 
\[
\Psi(0)(x, y) = y, \ \Psi(1)(x, y) = (y, x+y).
\] 
Then the multispinal group $G(A, B, \Psi)$ coincides with the Grigorchuk group generated by $a, b, c, d$ (we use the same symbols as in Example \ref{Gri}).
To see this,
it is sufficient to identify $a$ with the generator $1 \in \Z_2$ and identify
$b, c, d$ with $(0, 1), (1, 1), (1, 0) \in \Z_2 \times \Z_2$, 
respectively.
\end{exa}

From now on,
$e$ is always the unit of a multispinal group $G$.

\begin{lem}\label{lemsec5}
Let $G=G(A, B, \Psi)$ be a multispinal group over $X$.
Take $w \in X^\om, g, g' \in A\cup B \subset G$ with $g(w) = g'(w)$.
We write 
\[
w =x_1x_2x_3 \cdots ; \quad x_1, x_2, x_3, \ldots \in X . 
\]
Assume that $w \in X^\om$ uses an alphabet in $Y$ and define
\[
m := \min \{ i : x_i \in Y \}.
\]  
Then $g$ and $g'$ coincide on the cylinder set $w^{(m)}X^\om$.
In particular,
we obtain
$[g, w] = [g', w]$. 
\end{lem}

\begin{proof}
First we assume that $g \in B \backslash \{e \}$.
Then by $g(w) = g'(w)$, 
we have $g=g'$ (otherwise the first alphabet of $g(w)$ and $g'(w)$ do not coincide).

Second, let $g \in A$.
By the first step, we may assume $g' \in A$. 
Thus,
it is sufficient to show that $g$ and $e$ coincide on $w^{(m)}X^\om$ for $g \in A$ with $g(w) =w$. 
By definition,
$x_m\in Y$ and $x_i \notin Y$ for $1 \leq i \leq m-1$.
Hence, 
\[
( \Psi(x_m) \circ \Psi(x_{m-1}) \circ \cdots \circ \Psi(x_1) )(g) \in B.
\]  
Then,
since $g(w) =w$,
we obtain
\[
g(w^{(m)}) = w^{(m)}, \
g|_{w^{(m)}} = ( \Psi(x_m) \circ \Psi(x_{m-1}) \circ \cdots \circ \Psi(x_1) ) (g) =e.
\]  
This shows the claim.  
\end{proof}

\begin{lem}\label{lemsec52}
Let $G=G(A, B, \Psi)$ be a multispinal group over $X$.
Then we have $(X \backslash Y)^\om \subset X^\om \backslash (X^\om)_{G\textup{-reg}}$. 
More specifically,
we have
\[
a(w) = a'(w), \ [a, w] \neq [a', w]
\]  
for any $w \in (X \backslash Y)^\om, a, a' \in A$ with $a \neq a'$. 
\end{lem}

\begin{proof}
It suffices to show that
\[
a(w) = w, \ [a, w] \neq [e, w]
\]  
for $w \in (X \backslash Y)^\om, a \in A \backslash \{e \}$.
We write
\[
w =x_1x_2x_3 \cdots ; \quad x_1, x_2, x_3, \ldots \in X.
\] 
By the definition of the action $A \acts X^\om$,
$a$ does not change $x_1$.
The assumption $w \in (X \backslash Y)^\om$ implies $x_1 \in X \backslash Y$.
Thus,
$\Psi(x_1)(a) \in A$.
This shows that $\Psi(x_1)(a)$ does not change $x_2 \in X \backslash Y$.
Repeating this,
we have $a(w) =w$.

For each $i$,
$\Psi(x_i)$ is an automorphism on $A$. 
Therefore,
we obtain 
\[
a|_{w^{(n)}} = (\Psi(x_n) \circ \Psi(x_{n-1}) \circ \cdots \circ \Psi(x_1))(a) \neq e
\]
for any $n \in \N$.
Thus $[a, w] \neq [e, w]$.
\end{proof}

Let $G =G(A, B, \Psi)$ be a multispinal group over $X$.
Recall that $G$ is contracting \cite[Proposition 7.1]{SS}.
Hence we have a faithful state $\psi_{\min}$ on $\mO_{G_{\min}}$ by Proposition \ref{cont} and Theorem \ref{main2}.
We write 
\[
\pi_{\psi_{\min}} \colon \mO_{G_{\min}} \to B(H_{\psi_{\min}}), \quad \pi_{\psi_{\max}} \colon \mO_{G_{\max}} \to B(H_{\psi_{\max}})
\] 
for the GNS representations of $\psi_{\min}$ and $\psi_{\max}$,
respectively.
Let $w$ be an arbitrary strictly $G$-regular point.
Note that we have $\psi_{\max}(\eta)=\psi_{\min}(\pi_w(\eta))$ for any $\eta \in \mO_{G_{\max}}$.
Then
we obtain a unitary $U$ from $H_{\max}$ onto $H_{\min}$ given by
\[
U(\hat{\eta}):=\widehat{\pi_w(\eta)}
\]
for $\eta \in \mO_{G_{\max}}$.
Here $\hat{\eta} \in H_{\max}$ and $\widehat{\pi_w(\eta)} \in H_{\min}$ are equivalent classes represented by $\eta$ and $\pi_w(\eta)$,
respectively.
The unitary $U$ provides a unitarily equivalence between $(\pi_{\psi_{\min}} \circ \pi_w, H_{\psi_{\min}})$ and $(\pi_{\psi_{\max}}, H_{\psi_{\max}})$. 
In addition,
$\pi_{\psi_{\min}}$ is faithful since $\mO_{G_{\min}}$ is simple by Theorem \ref{simple}.
Thus,
we have  
\[
\| \pi_w(\eta) \|= \| \pi_{\psi_{\min}}(\pi_w(\eta)) \| = \| \pi_{\psi_{\max}}(\eta) \|
\]
for any $\eta \in \mO_{G_{\max}}$.

Write $C^*(A)$ for the full group \AL of $A$.
Since $A$ is finite,
the canonical inclusion 
\[
A \to \mO_{G_{\max}}
\]
extends to an embedding 
\[
C^*(A) \to \text{span} (A) \subset \mO_{G_{\max}}.
\] 
From now on, 
we identify $C^*(A)$ with the finite dimensional C$^*$-subalgebra $\text{span} (A)$ of $\mO_{G_{\max}}$.
Note that the restriction of $\psi_{\max}$ to $\mO_{G_{\max}} ^\Ga$ induces a tracial state.

\begin{lem}\label{lininde}
Let $G=G(A, B, \Psi)$ be a multispinal group over $X$.
Assume that the restriction of $\psi_{\max}$ to $C^*(A)$ is faithful.
Then there exists an orthonormal system $\{\overline{a} \}_{a \in A}$ in $H_{\psi_{\max}}$ with $\pi_{\psi_{\max}}(a) \overline{a'} = \overline{aa'}$ for any $a, a' \in A$.
\end{lem}

\begin{proof}
Let $\tau$ be the canonical tracial state on $C^*(A)$:
\[
\tau(a) = \de_{a, e}
\]
for any $a \in A$.
Here $\de$ is the Kronecker delta.
Recall that $\tau$ is faithful.
In addition,
the restriction of $\psi_{\max}$ to $C^*(A)$ is a faithful tracial state on $C^*(A)$ by assumption.
Note that $C^*(A)$ is finite dimensional.
Therefore, 
the non-commutative Radon--Nikodym theorem provides a positive invertible element $h$ in $C^*(A)$ with 
\[
\tau(k) =\psi_{\max}(h k )
\]
for any $k \in C^*(A)$ (see \cite[Corollary 3.3.6]{T2}).
Define
\[
\overline{a} := ah^{\frac{1}{2}}
\]
for each $a \in A$.
Note that the restriction of the canonical map $\mO_{G_{\max}} \to H_{\psi_{\max}}$ to $C^*(A)$ is injective by assumption.
We regard each $\overline{a}$ as a unit vector in $H_{\psi_{\max}}$ via the canonical map.
Clearly,
\[
\pi_{\psi_{\max}}(a) \overline{a'} = \overline{aa'}
\] 
for any $a, a' \in A$.
Considering the inner product $\langle \cdot, \cdot \rangle$ on $H_{\psi_{\max}}$,
we have
\[
\langle \overline{a}, \overline{e} \rangle
=\psi_{\max}(ah)=\psi_{\max}(ha)=\tau(a)=\de_{a, e}
\]
for any $a \in A$.
This shows the claim.
\end{proof}

We fix notations for the proof of the main theorem of this section.
We observe the reduced groupoid \AL $C^*_r([G, X]^\Ga)$.
Consider the map
\[
w \mapsto [e, w]
\]
from $X^\om$ to $[G, X]^\Ga$.
It is not hard to show that the map provides a homeomorphism from $X^\om$ onto the unit space $([G, X]^\Ga)^{(0)}$.
For $[e, w] \in ([G, X]^\Ga)^{(0)}$,
we write $w$ instead of $[e, w]$.
Then the C$^*$-norm on $C^*_r([G, X]^\Ga)$ is given by 
\[
\sup_{w \in X^\om} \| \la^{\Ga}_w(\, \cdot \, ) \|.
\]
Let $\C \langle G, X \rangle ^\Ga$ be the $*$-subalgebra of $\mO_{G_{\max}}$ generated by $\langle G, X \rangle ^\Ga$.
Note that for each $w \in X^\om$,
we have
\[
([G, X]^\Ga)_w = \{ [f, w] : f \in \langle G, X \rangle ^\Ga, w \in \text{Dom}f \}. 
\]
Let $\de_{[f, w]} \in l^2(([G, X]^\Ga)_w)$ be the characteristic function of $[f, w] \in ([G, X]^\Ga)_w$.
Then we have 
\[
\la_w^\Ga(f_1) \de_{[f_2, w]} = \begin{cases}
    \de_{[f_1f_2, w]} & \text{if} \ f_2(w) \in \text{Dom}f_1 \\
    0 & \text{otherwise}
  \end{cases}
\]
for any $f_1 \in \langle G, X \rangle^\Ga$ and $[f_2, w] \in [G, X]^\Ga$.

\begin{rem}\label{remsixone}
Let $G$ be a self-similar group over $X$.
For any $S_ugS_v^* \in \langle G, X \rangle ^\Ga$ with $|u| = |v| <n_0$,
we have

\begin{equation}
\label{selfsimspan}
g = g \sum_{u' \in X^{n_{0} -|u|}} S_{u'}S_{u'}^* = \sum_{u' \in X^{n_{0} -|u|}} S_{g(u')}g|_{u'} S_{u'}^*.
\end{equation}

Consider any finite sum
\begin{equation}
\label{selfsimspan2}
\eta = \sum_{g, u, v, |u| =|v|} \al_{u, v}^g S_u g S_v^* \in \C \langle G, X \rangle ^\Ga.
\end{equation}
Using (\ref{selfsimspan}),
we may assume that the sum is taken over $u, v \in X^n$ for some fixed $n \in \N$.  

Next,
we further assume that $G$ is a contracting self-similar group with the nucleus $\mN$.
Then for any finite subset $F' \subset G$,
there exists a natural number $n_1 \in \N$ satisfying $g'|_{u'} \in \mN$ for any $g' \in F'$ and $u' \in X^*$ with $|u'| >n_1$.
Thus, for any $g' \in F', n' >n_1$,
we have
\[
g' = g' \sum_{u' \in X^{n'}} S_{u'}S_{u'}^* = \sum_{u' \in X^{n'}} S_{g'(u')}g'|_{u'} S_{u'}^*.
\]
This shows that one can further assume that the sum in (\ref{selfsimspan2}) is taken over $g \in \mN, u, v \in X^n$ for some fixed (larger) $n \in \N$.
\end{rem}

For finite words $u, v \in X^*$, 
$uv$ is also a finite word.
For short, 
we temporary write
\[
u^{-1} (uv) := v.
\]
Let $w \in X^\om$ and $n, m$ be natural numbers with $n > m$.
We write
\[
w^{-(m) + (n)} := (w^{(m)})^{-1} w^{(n)} \in X^{n-m}. 
\]

\begin{rem}\label{remsixtwo}
Let $G$ be a self-similar group over $X$.
Take any $w \in X^\om$. 
For any $[f, w] \in ([G, X]^\Ga)_w$,
there exist $u \in X^*, g \in G$ and $m \in \N$ with $|u|=m, f=S_ugS^* _{w^{(m)}}$.
For any natural number $n$ with $n >m$,
we have
\begin{equation}
\label{selfsimgerm}
[S_u g S_{w^{(m)}} ^*, w] = [S_ugS_{w^{(m)}}^* S_{w^{(n)}}S_{w^{(n)}}^*, w] = [S_uS_{g(w^{-(m) + (n)})}g|_{w^{-(m)+(n)}}S_{w^{(n)}} ^*, w].
\end{equation}

Assume that $G$ is a contracting self-similar group with the nucleus $\mN$.
Take finitely supported function $\xi \in l^2(([G, X]^\Ga)_w)$.
Then, 
by the equation (\ref{selfsimgerm}) and similar arguments to Remark \ref{remsixone},
one can write  
\begin{equation}
\label{selfsimgermspan}
\xi = \sum_{u' \in X^{n'}, g' \in \mN} \beta_{u'} ^{g'} \de_{[S_{u'}g'S_{w^{(n')}} ^*, w]}
\end{equation}
for some fixed natural number $n'$ and complex numbers $\beta_{u'} ^{g'}$.
Note that one can replace $n'$ to be arbitrarily large. 
\end{rem}

In the following theorem,
we assume that $[G, X]$ is amenable.
See Theorem \ref{ameimpame} and Remark \ref{multiassume} for the sufficient condition of the amenability of $[G, X]$.

\begin{thm}\label{themsec5}
Let $G=G(A, B, \Psi)$ be a multispinal group over $X$.
Assume that $[G, X]$ is amenable.
Then the following conditions are equivalent.
\begin{enumerate}
\renewcommand{\labelenumi}{\textup{(\roman{enumi})}}
\item
The restriction of $\psi_{\max}$ to $C^*(A)$ is faithful.
\item
The $|A| \times |A|$ matrix $[\psi(a_1^{-1} a_2)]_{a_1, a_2 \in A}$ is invertible.
\item
The universal \AL $\mO_{G_{\max}}$ is simple.
\end{enumerate}
\end{thm}

\begin{proof}
The equivalence between (i) and (ii) is given by standard arguments.

If we assume (iii), 
then $\pi_{\psi_{\max}}$ is faithful by the simplicity of $\mO_{G_{\max}}$.
Recall that $\psi_{\max}$ is a KMS state on $\mO_{G_{\max}}$.
Therefore,
$\psi_{\max}$ is faithful by \cite[Corollary 5.3.9]{BR}.
This proves (iii) $\Rightarrow$ (i).

We show (i) $\Rightarrow$ (iii).
Combining the amenability of $[G, X]$ and Theorem \ref{isomNek1},
we have the identifications 
\[
\mO_{G_{\max}} \simeq C^*([G, X]) \simeq C^*_r([G, X]).
\]
We prove $\mO_{G_{\min}} \simeq C_r^*([G, X])$.
Recall that $\mO_{G_{\min}}$ is a simple quotient of $C_r^*([G, X])$ (Theorems \ref{uniqueNek}, \ref{simple}).
It is sufficient to show the injectivity of this quotient map.
Since the gauge action $\Ga$ is periodic,
we only check the injectivity on the gauge invariant subalgebra $C_r^*([G, X])^\Ga \simeq C_r^*([G, X]^\Ga)$ (see \cite[Proposition 4.5.1]{BO}).
We do this by the norm estimate.

Take any $w \in X^\om$ and any nonzero element $\eta \in \C \langle G, X \rangle^\Ga$.
The goal of this proof is to show
\[
\| \la^\Ga_w(\eta) \| \leq \| \pi_{\psi_{\max}} (\eta) \| = \|\pi_{w_0}(\eta) \|.
\]
Here 
$w_0$ is a strictly $G$-regular point which we choose later. 
Recall that we have $\| \pi_{\psi_{\max}} (\eta) \| = \| \pi_{w'}(\eta) \|$ for any $w' \in \Ggen ^{\text{S}}$.
From now on,
we write $\psi$ for $\psi_{\max}$.

By the definition of operator norm,
for any $\ve >0$
one can choose a nonzero finitely supported function $\xi \in l^2([G, X]^\Ga_w)$ with
\[
\| \la^\Ga_w(\eta)\xi \| \geq (1-\ve) \| \la^\Ga_w(\eta) \| \|\xi \|.
\]
Write
\[
\eta = \sum_{u, v, g, |u| =|v|} \al_{u, v}^g S_u g S_v^*, \quad \xi = \sum_{u', g'} \be_{u'}^{g'} \de_{[S_{u'}g'S^*_{w^{(|u'|)}}, w]}.
\]
Here $\displaystyle \de_{[S_{u'}g'S^*_{w^{(|u'|)}}, w]}$ is the characteristic function on $[S_{u'}g'S^*_{w^{(|u'|)}}, w]$.
Note that $G$ is contracting \cite[Proposition 7.1]{SS}.
We write $\mN$ for the nucleus of $G$,
namely
\[
\mN := A \cup \bigcup_{y \in Y} \Psi(y)(A) \subset A \cup B.
\]  
By Remarks \ref{remsixone}, \ref{remsixtwo},
we may assume that the finite sums are taken over $g, g' \in \mN$ and $u, v, u' \in X^n$ for some fixed $n \in \N$. 
We write 
\[
w =x_1x_2x_3 \cdots ; \quad x_1, x_2, x_3, \ldots \in X.
\]
 
We divide the proof into two cases.
For the first case, 
assume that $w$ uses infinitely many alphabets in $Y$.
Then we have a natural number $n'$ with $n' >n$ such that there exists $i \in \N$ with $n < i < n'$ and $x_i \in Y$.
Note that we have 
\[
bS_{x} = S_{b(x)}, \ aS_{yx} = S_{y}S_{(\Psi(y)(a))(x)}
\]
for any $x \in X, y \in Y, a \in A, b \in B$.
This shows
\begin{equation*}
\begin{split}
[S_{u'}g'S^* _{w^{(n)}}, w] 
&= [S_{u'}g'S_{w^{(n)}}^* S_{w^{(n')}}S_{w^{(n')}}^*, w] \\
&= [S_{u'}S_{g'(w^{-(n)+(n')})}g'|_{w^{-(n)+(n')}}S_{w^{(n')}} ^*, w] \\
&= [S_{u'}S_{g'(w^{-(n)+(n')})}S_{w^{(n')}} ^*, w]
\end{split}
\end{equation*}
for any $u' \in X^n, g' \in \mN \subset A \cup B$.
Now replace $n$ by $n'$ and rewrite 
\[
\eta = \sum_{g \in \mN, u, v \in X^n} \al_{u, v}^g S_u g S_v^*, \quad \xi = \sum_{u' \in X^n} \be_{u'} \de_{[S_{u'}S^*_{w^{(n)}}, w]}.
\]
Then, 
we compute
\[
\la^\Ga_w(\eta)\xi = \sum_{g \in \mN, u, v , u' \in X^n} \al_{u, v}^g \be_{u'} \la_w ^\Ga(S_ugS_v^*) \de_{[S_{u'}S^*_{w^{(n)}}, w]} = \sum_{g, u, v} \al_{u, v}^g \be_{v} \de_{[S_{u}gS^*_{w^{(n)}}, w]}.
\]
Let $m$ be the smallest natural number with $x_{n+m} \in Y$.
Using the density of $\Ggen$ in $X^\om$,
we choose a $G$-regular point $w_0 \in w^{(n+m)}X^\om$.
Define a vector $\tilde{\xi} \in l^2(\langle G, X \rangle(w_0))$ by
\[
\tilde{\xi} := \sum_{u' \in X^n} \be_{u'} \de_{(S_{u'}S_{w^{(n)}} ^*(w_0))}.
\]
A calculation shows
\[
\pi_{w_0}(\eta)\tilde{\xi} = \sum_{g \in \mN, u, v , u' \in X^n} \al_{u, v}^g \be_{u'} \pi_{w_0}(S_ugS_v^*)\de_{(S_{u'}S^*_{w^{(n)}}(w_0))} = \sum_{g, u, v} \al_{u, v}^g \be_{v} \de_{(S_{u}gS^*_{w^{(n)}}(w_0))}.
\]  
Note that $\left\{ \de_{[S_{u'}S_{w^{(n)}} ^*, w]} \right\} _{u' \in X^n}$ and $\left\{ \de_{(S_{u'}S_{w^{(n)}} ^*(w_0)} \right\} _{u' \in X^n}$ are orthonormal systems in $\l^2(([ G, X ]^\Ga)_{w})$ and $l^2(\langle G, X \rangle(w_0))$,
respectively.
Thus, 
we have $\| \xi \| = \| \tilde{\xi} \|$.
To compare the norms $\| \la^\Ga_w(\eta)\xi \|$ and $\| \pi_{w_0}(\eta) \tilde{\xi} \|$,
we prove the following claim.

\begin{claim}
Take any $u_1, u_2 \in X^n, g_1, g_2 \in A \cup B$.
Then the following conditions are equivalent.
\begin{enumerate}
\renewcommand{\labelenumi}{\textup{(\alph{enumi})}}
\item
$[S_{u_1}g_1S^*_{w^{(n)}}, w] = [S_{u_2}g_2S^*_{w^{(n)}}, w]$.
\item
$S_{u_1}g_1S^*_{w^{(n)}}(w_0) = S_{u_2}g_2S^*_{w^{(n)}}(w_0)$.
\end{enumerate}
\end{claim}

\begin{proof}[Proof of Claim]
First, 
we assume (a).
Then we have $u_1 = u_2$ and $g_1S_{w^{(n)}} ^*(w) = g_2S_{w^{(n)}} ^*(w)$.
Hence,
Lemma \ref{lemsec5} implies that $g_1S_{w^{(n)}} ^*$ and $g_2S_{w^{(n)}} ^*$ coincide on $w^{(n+m)}X^\om$ from the choice of $m$.
Since $w_0 \in w^{(n+m)}X^\om$,
we obtain (b).

Conversely,
we assume (b).
Then $u_1 = u_2$ and $g_1S_{w^{(n)}} ^*(w_0) = g_2S_{w^{(n)}} ^*(w_0)$.
Hence, 
Lemma \ref{lemsec5} shows that $g_1S_{w^{(n)}} ^*$ and $g_2S_{w^{(n)}} ^*$ coincide on $w_0 ^{(n+m)}X^\om = w^{(n+m)}X^\om$.
This implies (a).
\end{proof}

The claim implies $\| \la^\Ga_w(\eta)\xi \| = \| \pi_{w_0}(\eta) \tilde{\xi} \|$. 
Thus,
\begin{equation*}
\| \pi_{w_0}(\eta) \| \|\tilde{\xi} \| \geq \| \pi_{w_0} (\eta)\tilde{\xi} \| = \| \la^\Ga_w(\eta)\xi \| 
\geq  (1-\ve) \| \la^\Ga_w(\eta) \| \|\xi \| = (1-\ve) \| \la^\Ga_w(\eta) \| \|\tilde{\xi} \|.
\end{equation*}
Since $\ve$ is arbitrary,
we have $\| \pi_{w_0}(\eta) \| \geq \| \la^\Ga_w(\eta) \|$. 

For the second case, 
we assume that $w$ uses at most finitely many alphabets in $Y$.
Choose a natural number $n'$ with $S^*_{w^{(n')}}(w) \in (X \backslash Y)^\om$ and $n' > n$.
Then we have
\begin{equation*}
\begin{split}
[S_{u'}g'S^* _{w^{(n)}}, w] 
&= [S_{u'}g'S_{w^{(n)}}^* S_{w^{(n'+2)}}S_{w^{(n'+2)}}^*, w] \\
&= [S_{u'}S_{g'(w^{-(n)+(n'+2)})}g'|_{w^{-(n)+(n'+2)}}S_{w^{(n'+2)}} ^*, w]
\end{split}
\end{equation*}
and $g'(x_{n'+2}) \in X \backslash Y, g'|_{w^{-(n)+(n'+2)}} \in A$
for any $g' \in \mN \subset A \cup B$.
For each $l \in \N$,
define
\[
Z_l := \{ u =z_1z_2\cdots z_l \in X^{l} : z_l \in X \backslash Y \}.
\]   
Note that 
for any $g_0 \in A \cup B, u_0 \in Z_l$,
we have $g_0|_{u_0} \in A$.
Thus,
one can rewrite 
\[
\xi = \sum_{u' \in Z_{n'+2}, g' \in A} \beta_{u'} ^{g'} \de_{[S_{u'}g'S_{w^{(n'+2)}} ^*, w]}.
\]
Take an orthonormal system $\{ \overline{g} \}_{g \in A}$ given in Lemma \ref{lininde}. 
Define a vector $\tilde{\xi} \in H_\psi$ by
\[
\tilde{\xi} := \sum_{u' \in Z_{n'+2}, g' \in A} \be_{u'}^{g'} \pi_{\psi}(S_{u'})\overline{g'}.
\]
Take any $g' _1, g' _2 \in A, u' _1, u' _2 \in Z_{n'+2}$. 
By Lemma \ref{lemsec52}, 
$[S_{u' _1}g' _1S_{w^{(n'+2)}} ^*, w] = [S_{u' _2}g' _2S_{w^{(n'+2)}} ^*, w]$ if and only if $S_{u' _1} = S_{u' _2}, g' _1 =g' _2$.
Thus,
$\| \xi \| = \| \tilde{\xi} \|$.

For each
$u'  = z'_1z'_2 \cdots z'_{n'+2} \in Z_{n'+2}$
we define
\[
{u' }^{(n)} := z'_1z'_2 \cdots z'_n \in X^n.  
\]
Then for any $g \in \mN, g' \in A, u, v \in X^n, u' \in Z_{n'+2}$ with ${u'}^{(n)} = v$,
we have
\begin{equation*}
\begin{split}
\la^\Ga_w(S_ugS_v ^*) \de_{[S_{u'}g'S_{w^{(n'+2)}} ^*, w]}
&= \sum_{v' \in X^{n'+2-n}} \la^\Ga_w(S_ug S_{v'}S_{v'} ^*S_v ^*) \de_{[S_{u'}g'S_{w^{(n'+2)}} ^*, w]} \\
&= \de_{[S_{ug(v^{-1} u')}g|_{v^{-1} u'}g'S_{w^{(n'+2)}} ^*, w]}.
\end{split} 
\end{equation*}
Note that $g|_{v^{-1} u'} \in A$ since $v^{-1} u' \in Z_{n'+2-n}$.
Similarly,
we obtain
\[
\pi_{\psi}(S_ugS_v^*)\left(\pi_{\psi}(S_{u'})\overline{g'}\right) = \pi_{\psi}(S_{ug(v^{-1} u')})\overline{g|_{v^{-1} u'}g'}.
\]  
In addition,
for any $g \in \mN, g' \in A, u, v \in X^n, u' \in Z_{n'+2}$ with ${u'}^{(n)} \neq v$,
we have
\[
\la^\Ga_w(S_ugS_v ^*) \de_{[S_{u'}g'S_{w^{(n'+2)}} ^*, w]} = 0, \quad \pi_{\psi}(S_ugS_v^*)\left(\pi_{\psi}(S_{u'})\overline{g'}\right)=0.
\] 
Hence,
we get
\[
\la^\Ga_w(\eta)\xi = \sum_{u, v, u', g, g'} \al_{u, v}^g \be_{u'}^{g'} \de_{[S_{ug(v^{-1} u')}g|_{v^{-1} u'}g'S_{w^{(n'+2)}} ^*, w]}. 
\]
and 
\[
\pi_\psi(\eta)\tilde{\xi} = \sum_{u, v, u', g, g'} \al_{u, v}^g \be_{u'}^{g'} \pi_{\psi}(S_{ug(v^{-1} u')})\overline{g|_{v^{-1} u'}g'}.
\]
Here the sums are taken over $g \in \mN, g' \in A, u, v \in X^n, u' \in Z_{n'+2}$ with ${u'}^{(n)} = v$.

We claim $\| \la^\Ga_w(\eta) \xi \| = \| \pi_\psi(\eta) \tilde{\xi} \|$.
Define
\[
T := \{ S_{v'}h : h \in A, v' \in Z_{n'+2} \} \subset \langle G, X \rangle.
\] 
Take any $S_{v'_1}h_1, S_{v'_2}h_2 \in T$.
By Lemma \ref{lemsec52}, 
$[S_{v'_1}h_1S^* _{w^{(n'+2)}}, w] = [S_{v'_2}h_2 S^* _{w^{(n'+2)}}, w]$ if and only if $S_{v'_1} =S_{v'_2}, h_1 = h_2$.
Hence,
$[S_{v'_1}h_1 S^* _{w^{(n'+2)}}, w] = [ S_{v'_2}h_2S^* _{w^{(n'+2)}}, w]$ if and only if $\pi_{\psi}(S_{v'_1})\overline{h_1} =\pi_{\psi}(S_{v'_2})\overline{h_2}$. 
Now,
we rewrite 
\begin{equation}
\label{maineq61}
\la^\Ga_w(\eta)\xi = \sum_{S_{v'}h \in T} \left( \sum_{S_{ug(v^{-1} u')}g|_{v^{-1} u'}g' = S_{v'}h} \al_{u, v}^g \be_{u'}^{g'} \right) \de_{[S_{v'}hS_{w^{(n'+2)}} ^*, w]}, 
\end{equation}
\begin{equation}
\label{maineq62}
\pi_\psi(\eta)\tilde{\xi} = \sum_{S_{v'}h \in T} \left( \sum_{S_{ug(v^{-1} u')}g|_{v^{-1} u'}g' = S_{v'}h} \al_{u, v}^g \be_{u'}^{g'} \right) \pi_{\psi}(S_{v'})\overline{h}. 
\end{equation}
The equations (\ref{maineq61}, \ref{maineq62}) shows the claim.
Consequently,
we have $\| \pi_{\psi}(\eta) \| \geq \| \la^\Ga_w(\eta) \|$ from the same argument as the first case.
Thus we finished the proof.
\end{proof}

\begin{cor}
Let $G=G(A, B, \Psi)$ be a multispinal group over $X$.
Assume that $[G, X]$ is amenable.
If one of the equivalent conditions of Theorem \ref{themsec5} holds,
then the universal \AL $\mO_{G_{\max}}$ is a Kirchberg algebra.
\end{cor}

\begin{proof}
Since the assumption implies that $\mO_{G_{\max}}$ is simple,
we have
\[
\mO_{G_{\max}} \simeq \mO_{G_{\min}}.
\]
Thus,
$\mO_{G_{\max}}$ is purely infinite simple by Theorem \ref{simple}.
By assumption, 
$[G, X]$ is amenable.
Then as $C^*([G, X])$ is nuclear, so is $\mO_{G_{\max}}$.
Trivially,
$\mO_{G_{\max}}$ is separable and unital.
Hence,
$\mO_{G_{\max}}$ is a Kirchberg algebra. 
\end{proof}

For the other corollary,
we fix notations.
Consider a multispinal group $G=G(A, B ,\Psi)$.
We write $C^*(B)$ for the full group \AL of the finite group $B$.
In the same way as $C^*(A)$,
we identify $C^*(B)$ with a finite dimensional C$^*$-subalgebra of $\mO_{G_{\max}}$.
For each $\la \in \text{Hom}(A, B)$,
define a $*$-homomorphism $\tilde{\la} \colon C^*(A) \to C^*(B)$ to be
\[
\tilde{\la}\left(\sum_{a \in A} \al_a a \right) := \sum_{a \in A} \al_a \la (a).
\]
Here each $\al_a$ is a complex number.
We recall a spacial case of \cite[Theorem 7.5]{SS}.

\begin{thm}\label{SSthm}\textup{(\cite[Theorem 7.5]{SS})}
Let $G=G(A, B, \Psi)$ be a multispinal group over $X$.
Then $\C \langle G, X \rangle$ is simple if and only if 
\[
\bigcap_{\la \in \mB \cdot \mA} \ker \tilde{\la} = \{ 0 \}.
\]
\end{thm}

Combining Theorems \ref{themsec5}, \ref{SSthm},
we have the following corollary.

\begin{cor}
Let $G=G(A, B, \Psi)$ be a multispinal group over $X$.
Assume that $[G, X]$ is amenable.
Then $\mO_{G_{\max}}$ is simple if and only if $\C \langle G, X \rangle$ is simple.
\end{cor}

\begin{proof}
First, 
we show the simplicity of $\C \langle G, X \rangle$ from the simplicity of $\mO_{G_{\max}}$.
In sections 2, 4 of \cite{SS},
it was shown that $\C \langle G, X \rangle$ is isomorphic to the Steinberg $\C$-algebra associated to $[G, X]$.
Thus,
the simplicity of $\mO_{G_{\max}} \simeq C^*([G, X])$ implies the simplicity of $\C \langle G, X \rangle$ by \cite[Corollary 4.12]{CE}.  

Next,
we prove the simplicity of $\mO_{G_{\max}}$ from the simplicity of $\C \langle G, X \rangle$.
We check the condition (i) of Theorem \ref{themsec5}.
Take any nonzero positive element $\eta \in C^*(A)$.
Write
\[
\eta = \sum_{a \in A} \al_a a.
\] 
Since $\eta \neq 0$,
there exists a homomorphism $\la \in \mB \cdot \mA$ with $\tilde{\la}(\eta) \neq 0$ by Theorem \ref{SSthm}.
In the case $\la \notin \mB$,
write
\[
\la =\Psi(x_n) \circ \Psi(x_{n-1}) \circ \cdots \circ \Psi(x_1), \ u=x_1x_2 \cdots x_n .
\]
Here $n \geq 2, x_n \in Y$ and $x_i \in X \backslash Y$ for any $1 \leq i \leq n-1$.
In the case $\la \in \mB$,
we write $\la = \Psi(x_1)$ and $u =x_1 \in Y$.  
For each $a \in A$,
we have 
\begin{equation*}
a S_u = a S_{x_1}S_{x_2} \cdots S_{x_n} = S_{x_1} \Psi(x_1)(a) S_{x_2} \cdots S_{x_n} = \cdots = S_u \la(a).
\end{equation*}
This shows that
\begin{equation}\label{lastcor1}
\eta S_uS^* _u = \left( \sum_{a \in A} \al_a a \right) S_uS_u ^* = S_u \left( \sum_{a \in A} \al_a \la(a) \right) S_u ^* = S_u \tilde{\la}(\eta) S_u ^*.
\end{equation}
Note that $\tilde{\la}(\eta)$ is a nonzero positive element in $C^*(B)$.
We write
\[
(\tilde{\la}(\eta))^{\frac{1}{2}} = \sum_{b \in B} \beta_b b.
\]
Here $\be_{b_0} \neq 0$ for some $b_0 \in B$.
Since $\psi_{\max}(b) = \mu((X^\om)_b) =0$ for any $b \neq e$,
we have
\begin{equation}\label{lastcor2}
\psi_{\max}(\tilde{\la}(\eta)) = \psi_{\max}\left (\left(\sum_{b \in B} \beta_b b \right) ^* \left(\sum_{b \in B} \beta_b b \right)\right) = \sum_{b \in B} |\be_b |^2 > 0 .
\end{equation}
Recall that the restriction of $\psi_{\max}$ to $\mO_{G_{\max}} ^\Ga$ is tracial.
Then a calculation shows
\begin{equation}\label{lastcor3}
\psi_{\max}(\eta S_uS_u ^*) = \psi_{\max}(\eta ^\frac{1}{2} S_uS_u ^* \eta ^\frac{1}{2}) \leq \psi_{\max}(\eta ^\frac{1}{2} \eta ^\frac{1}{2} ) = \psi_{\max}(\eta). 
\end{equation} 
Combining (\ref{lastcor1}), (\ref{lastcor2}), (\ref{lastcor3}) and Theorem \ref{main2},
we obtain
\[
\psi_{\max}(\eta) \geq \psi_{\max}(\eta S_uS_u ^*) = \psi_{\max}(S_u \tilde{\la}(\eta) S_u ^*) = |X|^{-|u|}\psi_{\max}(\tilde{\la}(\eta)) > 0.
\]
This proves the claim.
\end{proof}

Using Theorem \ref{themsec5},
we show the simplicity of $\mO_{G_{\max}}$ in the case that $G$ is the Grigorchuk group.
The simplicity of $\mO_{G_{\max}}$ of the Grigorchuk group has been proved in \cite[Theorem 5.22]{CE}.
The proof in \cite{CE} is based on its main theorem. 
The main theorem provides a relevance between supports of functions on a non-Hausdorff groupoid and the simplicity of its reduced groupoid \AL (see \cite[Theorem 4.10]{CE}).
We observe a different proof in the following example.    

\begin{exa}\label{simplegri}
We regard the Grigorchuk group as a multispinal group (see Example \ref{ortgri}).
In this example,
we use the same symbols as in Example \ref{Gri} for generators.
In other words,
we write 
\[
\Z_2 \times \Z_2 = \{ e, b, c, d \}, \ \Z_2 = \{ e, a \}.
\]
Note that we have
\[
b^2 = c^2 = d^2 =e, bc =cb =d, bd =db =c, cd =dc =b.
\]
We determine the values $\psi(b), \psi(c), \psi(d)$ to check the condition (ii) in Theorem \ref{themsec5}.
This is done if one solves the following simultaneous equation given by the self-similarity:
\[
\psi(b) = \frac{\psi(a) + \psi(c)}{2} = \frac{\psi(c)}{2} ,
\] 
\[
\psi(c) = \frac{\psi(a) + \psi(d)}{2} = \frac{\psi(d)}{2} ,
\] 
\[
\psi(d) = \frac{\psi(e) + \psi(b)}{2} = \frac{1+ \psi(b)}{2}.
\] 
As the solution of these equations,
we have $\displaystyle \psi(b) = \frac{1}{7}, \psi(c) = \frac{2}{7}, \psi(d) = \frac{4}{7}$. 
Thus,
\[
\left(
    \begin{array}{cccc}
      \psi(e) & \psi(b) & \psi(c) & \psi(d) \\
      \psi(b) & \psi(e) & \psi(d) & \psi(c) \\
      \psi(c) & \psi(d) & \psi(e) & \psi(b) \\
      \psi(d) & \psi(c) & \psi(b) & \psi(e)
    \end{array}
  \right)
= \frac{1}{7}\left(
    \begin{array}{cccc}
      7 & 1 & 2 & 4 \\
      1 & 7 & 4 & 2 \\
      2 & 4 & 7 & 1 \\
      4 & 2 & 1 & 7
    \end{array}
  \right).
\]
Compute the determinant of the matrix
\[
\text{det}\left(
    \begin{array}{cccc}
      7 & 1 & 2 & 4 \\
      1 & 7 & 4 & 2 \\
      2 & 4 & 7 & 1 \\
      4 & 2 & 1 & 7
    \end{array}
  \right)
= 896.
\]
By Theorem \ref{themsec5}, 
the universal \AL $\mO_{G_{\max}}$ is simple.
\end{exa}

In the next example,
we observe a non-simple case.
Combining \cite[Corollary 4.12]{CE} and \cite[Example 7.6]{SS},
one can get the same result.

\begin{exa}
Let $A=\Z_2 \times \Z_2, B=\Z_2$ and $X =\{0, 1\}$.
Consider the left translation action $B \acts X$.
We define $\Psi(0) \in \text{Hom}(\Z_2 \times \Z_2, \Z_2)$ and $\Psi(1) \in \text{Aut}(\Z_2 \times \Z_2)$ to be 
\[
\Psi(0)(x, y) := y, \ \Psi(1)(x, y) := (y, x).
\]
Respecting the Grigorchuk group,
we write $b := (0, 1), c :=(1, 1), d: =(1, 0)$ and let $a$ denote the generator of $\Z_2$.
Then one obtains
\[
a(0w) = 1w , \ a(1w) = 0w, 
\]
\[
 b(0w) = 0a(w), \ b(1w) = 1d(w),
\]
\[
c(0w) = 0a(w) , \ c(1w) = 1c(w), 
\]
\[
 d(0w) = 0w , \ d(1w) = 1b(w)
\]
for any $w \in X^\om$.
These equations give
\[
\psi(b) = \frac{\psi(a) + \psi(d)}{2} = \frac{\psi(d)}{2} ,
\] 
\[
\psi(c) = \frac{\psi(a) + \psi(c)}{2} = \frac{\psi(c)}{2} ,
\] 
\[
\psi(d) = \frac{\psi(e) + \psi(b)}{2} = \frac{1+ \psi(b)}{2}.
\] 
Then we get
$\displaystyle \psi(b) = \frac{1}{3}, \psi(c) = 0, \psi(d) = \frac{2}{3}$.
Thus,
we have
\[
\det\left(
    \begin{array}{cccc}
      \psi(e) & \psi(b) & \psi(c) & \psi(d) \\
      \psi(b) & \psi(e) & \psi(d) & \psi(c) \\
      \psi(c) & \psi(d) & \psi(e) & \psi(b) \\
      \psi(d) & \psi(c) & \psi(b) & \psi(e)
    \end{array}
  \right)
= \det\left(\frac{1}{3}\left(
    \begin{array}{cccc}
      3 & 1 & 0 & 2 \\
      1 & 3 & 2 & 0 \\
      0 & 2 & 3 & 1 \\
      2 & 0 & 1 & 3
    \end{array}
  \right)\right)
=0.
\]
Theorem \ref{themsec5} implies that the universal \AL $\mO_{G_{\max}}$ is not simple in this case.
\end{exa}

\begin{exa}
Next we observe the case $A = \Z_3 \times \Z_3, B=X=\Z_3$.
The action of $B$ on $X$ is the left translation action.
We regard the elements in $A$ as $\Z_3$-valued row vectors. 
We define $\Psi(0), \Psi(1) \in \text{Aut}(\Z_3 \times \Z_3)$ and $\Psi(2) \in \text{Hom}(\Z_3 \times \Z_3, \Z_3)$ to be 
  \[
\Psi(0) 
= \left( \begin{array}{cc}
1 & 1 \\
0 & 1
\end{array} \right), \
\Psi(1) 
= \left( \begin{array}{cc}
1 & 0 \\
1 & 1
\end{array} \right), \ 
\Psi(2) 
= \left( \begin{array}{cc}
0 & 1 
\end{array} \right).
\]
Let
\[
a_0 = 
\left( \begin{array}{c}
0  \\
0
\end{array} \right) , \
a_1 = 
\left( \begin{array}{c}
1  \\
0
\end{array} \right) , \
a_2 = 
\left( \begin{array}{c}
0  \\
1
\end{array} \right) , \
a_3 = 
\left( \begin{array}{c}
1  \\
1
\end{array} \right) , \
a_4 = 
\left( \begin{array}{c}
1  \\
2
\end{array} \right). 
\]
By the self-similarity,
we have
\[
\psi(a_1) = \frac{\psi(a_1) + \psi(a_3)+1}{3}, 
\]
\[
\psi(a_2) = \frac{\psi(a_3) + \psi(a_2)}{3}, 
\]
\[
\psi(a_3) = \frac{2 \psi(a_4)}{3}, 
\]
\[
\psi(a_4) = \frac{\psi(a_2) + \psi(a_1)}{3}. 
\]
From these equations,
we obtain
\[
\psi(a_1)=\psi(a^{-1}_1)=\frac{4}{7}, 
\]
\[
\psi(a_2)=\psi(a^{-1}_2)=\frac{1}{14}, 
\]
\[
\psi(a_3)=\psi(a^{-1}_3)=\frac{1}{7}, 
\]
\[
\psi(a_4)=\psi(a^{-1}_4)=\frac{3}{14}.
\]
Hence we obtain the matrix
\[
[\psi(a^{-1}_ia_j )]_{0\leq i, j \leq 8}
=\frac{1}{14}
\left(
    \begin{array}{ccccccccc}
      14&1&8&2&3&1&8	&2&3 \\
      1&14&3&8&2&1&2&3&8 \\
      8&3&14&1&1&2&8&3&2 \\
      2&8&1&14&1&3&3&2&8 \\
      3&2&1&1&14&8&2&8&3 \\
      1&1&2&3&8&14&3&8&2 \\
      8&2&8&3&2&3&14&1&1 \\
      2&3&3&2&8&8&1&14&1 \\
      3&8&2&8&3&2&1&1&14 \\
    \end{array}
  \right).
\]
Here $a_5 :=a_1 ^{-1}, a_6 :=a_2 ^{-1}, a_7 :=a_3 ^{-1}, a_8 :=a_4 ^{-1}$. 
The MDETERM function of Excel provides the result of the calculation as follows: 
\[
\det\left(
    \begin{array}{ccccccccc}
      14&1&8&2&3&1&8	&2&3 \\
      1&14&3&8&2&1&2&3&8 \\
      8&3&14&1&1&2&8&3&2 \\
      2&8&1&14&1&3&3&2&8 \\
      3&2&1&1&14&8&2&8&3 \\
      1&1&2&3&8&14&3&8&2 \\
      8&2&8&3&2&3&14&1&1 \\
      2&3&3&2&8&8&1&14&1 \\
      3&8&2&8&3&2&1&1&14 \\
    \end{array}
  \right)
=634894848.
\]
As a consequence,
we get the simplicity of $\mO_{G_{\max}}$ in this case. 
\end{exa}
\subsection*{Acknowledgements}
The author appreciates his supervisors,
Yuhei Suzuki and Reiji Tomatsu, for fruitful discussions and their constant encouragements.

\end{document}